\documentclass[11pt,reqno]{amsart}
\usepackage{amsmath, amssymb,graphicx,mathrsfs,enumerate}
\usepackage{mathtools}
\usepackage[all]{xy}
\usepackage{hyperref}
\usepackage{pgf,tikz}
\usetikzlibrary{cd} 
\usepackage{tikz-3dplot}

\newtheorem{thm}{Theorem}

\newtheorem{thmA}{Theorem}

\newtheorem{corA}[thmA]{Corollary}
\newtheorem{prop}[thm]{Proposition}
\newtheorem{lem}[thm]{Lemma}

\newtheorem{cor}[thm]{Corollary}
\newtheorem{ex}[thm]{Example}

\newcommand{\calC}{{\mathcal C}}
\newcommand{\calT}{{\mathcal T}}
\newcommand{\calX}{{\mathcal X}}

\newcommand{\calD}{{\mathcal D}}

\newcommand{\ad}{{\operatorname{ad}}}
\newcommand{\gr}{{\operatorname{gr}}}
\newcommand{\id}{{\operatorname{id}}}
\newcommand{\op}{{\operatorname{op}}}
\newcommand{\HH}{{\operatorname{HH}}}
\newcommand{\Hh}{{\operatorname{H}}}

\usepackage{eqparbox}

\newcommand{\Ob}{\operatorname{Ob}}

\newcommand{\Mor}{\operatorname{Mor}}
\newcommand{\Der}{\operatorname{Der}}
\newcommand{\Arr}{\operatorname{Arr}}
\newcommand{\Ker}{\operatorname{Ker}}
\newcommand{\Ima}{\operatorname{Im}}
\newcommand{\Inn}{\operatorname{Inn}}
\newcommand{\Char}{\operatorname{Char}}

\newcommand{\End}{\operatorname{End}}

\newcommand{\Hom}{\operatorname{Hom}}

\newcommand{\im}{\operatorname{Im}}

\begin{document}
\title[Hochschild and simplicial cohomology of some finite category algebras]{Hochschild cohomology of some finite category algebras as simplicial cohomology}
\author{I.-I. Simion}
\author{C.-C. Todea}

\address[I.-I. Simion]
        { Department of Mathematics\\
          Babeș-Bolyai University\\
          Str. Ploieşti 23-25, Cluj-Napoca 400157, Romania\\
          and
          Department of Mathematics\\
          Technical University of Cluj-Napoca\\
          Str. G. Bari\c tiu 25, Cluj-Napoca 400027, Romania}
        \email{iulian.simion@ubbcluj.ro}
\address[C.-C. Todea]
        { Department of Mathematics\\
          Technical University of Cluj-Napoca\\
          Str. G. Bari\c tiu 25, Cluj-Napoca 400027, Romania}
        \email{Constantin.Todea@math.utcluj.ro}

\maketitle

\begin{abstract}
  By a result of Gerstenhaber and Schack the simplicial cohomology ring $\Hh^{\bullet}(\calC,k)$ of a poset $\calC$ is isomorphic to the Hochschild cohomology ring $\HH^{\bullet}(k\calC)$ of the category algebra $k\calC$, where the poset is viewed as a category and $k$ is a field. Extending results of Mishchenko \cite{Mishchenko}, under certain assumptions on a category $\calC$, we construct a category $\calD$ and a graded $k$-linear isomorphism $\HH^{\bullet}(k\calC)\cong \Hh^{\bullet}(\calD,k)$. Interpreting the degree one cohomology, we also show how  the $k$-space of derivations on $k\calC$, graded by some semigroup, corresponds to the $k$-space of characters on $\calD$.
  \end{abstract}

\section{Introduction}

Let $k$ be a field and let $\calC$ be a finite category, i.e. a category in which the objects form a finite set and every set of morphisms is finite. By a result of Gerstenhaber and Schack \cite{Gerstenhaber_Schack_83}, if $\calC$ is a poset, then the simplicial cohomology $\Hh^{\bullet}(\calC,k)$ is isomorphic to the Hochschild cohomology $\HH^{\bullet}(k\calC)$ of the category algebra $k\calC$. The proof of this result is simplified in \cite{Gerstenhaber_Schack_85} with the notion of $R$-relative Hochschild cohomology for subrings $R$ of $k\calC$. By a recent result of Mishchenko \cite[Theorem 1]{Mishchenko}, if $\calC$ is a group, then $\HH^{\bullet}(k\calC)$ is isomorphic to the  cohomology $\Hh^{\bullet}(BGr,k)$ of the classifying space of some groupoid $Gr$. The proof of this result describes an explicit $k$-linear isomorphism. In \cite{Xu_HHH} Fei Xu constructs a counterexample to a conjecture of Snashall and Solberg \cite{Sna_Sol} by studying  a split surjective ring homomorphism $\HH^{\bullet}(k\calC)\rightarrow\Hh^{\bullet}(B\calC,k)$. The natural question which arises is: for which categories can the Hochschild cohomology of $k\calC$ be identified explicitly with the singular cohomology of the classifying space of some category? Due to the well-known isomorphism $\Hh^{\bullet}(\calC,k)\cong\Hh^{\bullet} (B\calC,k)$ (see \cite[\S5]{Webb_representations_cohomology}) this is equivalent to asking for the existence of a category $\calD$ such that $\HH^{\bullet}(k\calC)$ is isomorphic to the simplicial cohomology $\Hh^{\bullet}(\calD,k)$.

In this paper we replace the category $F(\calC)$ used in \cite{Xu_HHH} with the \emph{adjoint category} $F^{\ad}(\calC)$, generalizing the explicit construction of the groupoid $Gr$ from \cite{Mishchenko}. We denote by $\Ob(\calC)$ the set of objects in $\calC$ and by $\Mor(\calC)$ the set of morphisms in $\calC.$ Let $x_1, x_2\in \Ob(\calC).$ The category $F^{\ad}(\calC)$ has as  objects the subset $\End\calC$ of endomorphisms in $\Mor(\calC)$ and morphisms are defined in the following way. A morphism in  $F^{\ad}(\calC)$ from $a\in \End_{\calC}(x_1)$ to $b\in \End_{\calC}(x_2)$ is a morphism $g\in\Hom_{\calC}(x_1,x_2)$, making the following diagram commutative:
\begin{equation}
  \label{ad_morphism}
  \begin{tikzcd}
    x_1 \arrow[r,"g"] & x_2 \\
    x_1 \arrow[u,"a"]\arrow[r,"g"] & x_2\arrow[u,swap,"b"]
  \end{tikzcd}.
\end{equation}
If $\calC$ is a group then $F^{\ad}(\calC)$ is isomorphic to the groupoid $Gr$ constructed in \cite{Mishchenko} and if $\calC$ is a poset then $F^{\ad}(\calC)$ is isomorphic to $\calC$. 

In order to extend the construction in \cite{Mishchenko} we require certain conditions on $\calC$. The category $\calC$ is called \emph{left deterministic} if for any $b\in\End_{\calC}(x_2)$ and $g\in\Hom_{\calC}(x_1,x_2)$ there exists 
$a\in\End_{\calC}(x_1)$ such that the diagram \eqref{ad_morphism} commutes. Similarly, the category $\calC$ is called \emph{right deterministic} if for any $a\in\End_{\calC}(x_1)$ and $g\in\Hom_{\calC}(x_1,x_2)$ there exists 
$b\in\End_{\calC}(x_2)$ such that the diagram \eqref{ad_morphism} commutes. We call $\calC$ \emph{deterministic} if it is both left and right deterministic. A category $\calC$ is left (respectively right) cancellative if the composition of morphisms has the left (respectively right) cancellation property. We call $\calC$ \emph{cancellative} if it is both left and right cancellative. The notion of cancellative categories needed for our results is not new. Lawson and Wallis use it to generalize some of their results \cite{LW_cancellative} from monoids to categories. 
There is a natural right action of  the monoid $\End_{\calC}(x_1)$ on the set $\Hom_{\calC}(x_1,x_2)$, given by composition:
$$
\Hom_{\calC}(x_1,x_2)\times \End_{\calC}(x_1)\rightarrow\Hom_{\calC}(x_1,x_2),\quad (g,a)\mapsto g\circ a.
$$
We call $\calC$ \emph{rr-transitive} (right-regular) if for any objects $x_1,x_2\in\Ob\calC$ the right action of  the monoid $\End_{\calC}(x_1)$ on the set $\Hom_{\calC}(x_1,x_2)$ is transitive, see \cite[Section 2.1]{Stein_monoid} for more details about actions of monoids on sets. 

We are now in a position to state our first result.
\begin{thmA}
  \label{HHH_epi}
  If $\calC$ is a finite deterministic cancellative category then there is a surjective morphism of graded $k$-vector spaces $\calT^{\bullet}:\HH^{\bullet}(k\calC)\rightarrow\Hh^{\bullet}(F^{\ad}(\calC),k)$ between the Hochschild cohomology of the category algebra $k\calC$ and the simplicial cohomology of the category $F^{\ad}(\calC)$. Moreover, if $\calC$ is rr-transitive, then $\calT^{\bullet}$ is an isomorphism.
\end{thmA}

The map $\calT^{\bullet}$ is induced from an explicitly constructed cochain map and unifies the constructions in \cite{Gerstenhaber_Schack_83} and \cite{Mishchenko} in our context. Moreover, it is easy to describe a right inverse of $\calT^{\bullet}$, again, induced from an explicit cochain map (Propositions \ref{complex_map_X} and \ref{prop_X_section}). Examples of finite deterministic cancellative rr-transitive categories are groupoids and posets (Subsection \ref{sec:cancellative}), hence we obtain the following corollary.

\begin{corA}
  \label{groupoids_posets}
  If $\calC$ is a finite groupoid or a finite poset then the Hochschild cohomology of the category algebra $k\calC$ is isomorphic to the simplicial cohomology of the category $F^{\ad}(\calC)$.
\end{corA}

The isomorphism $\HH^{\bullet}(k\calC)\cong\Hh^{\bullet}(\calC,k)$ for $\calC$ a poset does not hold for general $\calC$ (see \cite[\S5]{Webb_representations_cohomology} for more details). For certain EI categories (every endomorphism is an isomorphism) which are called \textit{amalgams of groups and posets}, see \cite[Section 4]{Lodder}, Lodder describes $\Hh^{\bullet}(\calC,k)$ as a direct summand of $\HH^{\bullet}(k\calC)$. 

In Example \ref{ex6} we describe an EI  category $\calC$,  for which the map  $\calT^{\bullet}$ is an isomorphism, but $\calC$ is not a groupoid neither a poset. The category $\calC$ is not even an amalgam of groups and posets.

The results in \cite{Mishchenko} build on ideas developed by Arutyunov and Mishchenko in \cite{Arutyunov_Mishchenko} where they show that $\HH^{1}(k\calC)\cong \Hh^{1}(BF^{\ad}(\calC),k)$ in the case where $\calC$ is a group. By Theorem \ref{HHH_epi}, there is an isomorphism of the first cohomology groups $\HH^{1}(k\calC)\cong \Hh^{1}(F^{\ad}(\calC),k)$. However, the method in \cite{Arutyunov_Mishchenko} relies on a description of the derivations of $k\calC$ in terms of characters on $F^{\ad}(\calC)$ (see Section \ref{sec:ders_chars}) and the fact that $ \Der(k\calC)/\Inn(k\calC)\cong \HH^{1}(k\calC),$  where $\Der(k\calC)$ is the algebra of derivations of $k\calC$ and $\Inn(k\calC)$ is the algebra of inner derivations. In order to extend the description in \cite{Arutyunov_Mishchenko}, we introduce in Subsection \ref{subsec:gradder} a semigroup with zero, denoted by $S$ and explain how $k\calC $ is graded by $S$. We denote by $\Der^{\gr}(k\calC)$ the $k$-vector space of all $S$-graded derivations. Our second result extends the characterization from \cite{Arutyunov_Mishchenko} of $\Der(k\calC)$ for group algebras as follows.

\begin{thmA}
  \label{HH1H1}
  Let $\calC$ be a finite rr-transitive deterministic cancellative category. There is a $k$-linear isomorphism $\calT:\Der^{\gr}(k\calC)\rightarrow \Char(F^{\ad}(\calC))$ between graded derivations on $k\calC$ and characters on the category $F^{\ad}(\calC)$.
  \end{thmA}

We note that the above results can be given for any small category $\calC$ if we use a condition of finiteness of supports for the cochains appearing in the above cohomologies. In this paper we assume that $\calC$ is finite. 

In \cite{Gerstenhaber_Schack_83}, Gerstenhaber and Shack show how the simplicial cohomology of a locally finite simplicial complex can be interpreted as the Hochschild cohomology of some associative unital $k$-algebra. It is known that the simplicial cohomology agrees with the singular cohomology (which is better adapted to theory than to computation) of topological spaces which can be triangulated. It is possible to compute simplicial cohomology of a simplicial complex efficiently, hence it is better suited for calculations. In Theorem \ref{HHH_epi} we describe certain finite categories which have Hochschild cohomology isomorphic (as graded $k$-vector space) to the simplicial cohomology of some other categories, with trivial coefficients. The Hochschild cohomology is endowed with a very rich structure: it is a graded commutative algebra via the cup product; it has a graded Lie bracket of degree $-1$; with this bracket the Hochschild cohomology becomes a Gerstenhaber algebra. Recently, the second author introduced a Beilinson-Drinfeld structure (BD-structure) in group cohomology \cite{To} and showed that the Hochschild cohomology $\HH^{\bullet}(\mathbb{F}_3C_3)$, of the cyclic $3$ group algebra, is a BD-algebra.  BD-algebras are very similar to BV-algebras (Batalin-Vilkovisky), which have been intensively studied for Hochschild cohomology of various algebras. It would be interesting to investigate the extent to which the maps $\calT^{\bullet},\calX^{\bullet}$ can transport (preserve) all these structures between Hochschild cohomology and simplicial cohomology.

\bigskip

The paper is structured as follows: In Section \ref{sec:preliminaries} we recall some facts on Hochschild cohomology and cohomology of the nerve of a small category. Subsection \ref{sec:cancellative} discusses the concepts of deterministic, cancellative  and rr-transitive categories which are needed in Section \ref{Mishchenko_general} to generalize the construction in \cite{Mishchenko}. The proof of Theorem \ref{HHH_epi} is given in Section \ref{sec:relative_HH}. In Section \ref{sec:ders_chars} we describe characters on the category $F^{\ad}(\calC)$ and give the proof of Theorem \ref{HH1H1}.

\section{Preliminaries}
\label{sec:preliminaries}

For an object $x$ in $\calC$ we denote by $\id_x$ the identity endomorphism on $x$. For a morphism $a:x\rightarrow y$ in $\calC$ we denote by $s(a)$ the source $x$ of $a$ and by $t(a)$ the target $y$ of $a$. For any category $\calC$ and a field $k$ the category algebra $k\calC$ is the  $k$-vector space with basis the morphisms in $\calC$ and with product defined by
$$
ab=
\left\{
\begin{array}{ll}
  a\circ b, &\text{ if $s(b)=t(a)$}\\
  0, &\text{ otherwise}
  \end{array}.
\right.
$$
We shall use the notation $\otimes$ for $\otimes_k$ and, if $A,B$ are $k$-vector vector spaces, by $\Hom_{k}(A,B)$ we denote the $k$-vector space of all $k$-linear maps from $A$ to $B$.
\subsection{Hochschild cohomology complex}
\label{hochschild_complex}
The Hochschild cochain complex \cite[Definition 1.1.13]{With_book} for the algebra $k\calC$ is
$$
0\xrightarrow{\mathmakebox[1.5em]{}}
C^{0}(k\calC)\xrightarrow{\mathmakebox[1.5em]{\partial^0}}
C^{1}(k\calC)\xrightarrow{\mathmakebox[1.5em]{\partial^1}}
\dots\xrightarrow{\mathmakebox[1.5em]{\partial^{m-1}}}
C^{m}(k\calC)\xrightarrow{\mathmakebox[1.5em]{\partial^m}}
C^{m+1}(k\calC)\xrightarrow{\mathmakebox[1.5em]{\partial^{m+1}}}
\dots
$$
where $C^{0}(k\calC)=k\calC$ and for $m\geq 1$
$$
C^{m}(k\calC)=\left\{f:\underbrace{k\calC\otimes k\calC\otimes \cdots\otimes k\calC}_{m\text{ times}}\rightarrow k\calC\,  \mid \text{ $f$ is $k$-linear}\right\}.
$$For any $f\in C^{m}(k\calC)$ and any $a_1,\dots,a_{m+1}\in k\calC$ we have 
 
 $$\partial^{m}(f)(a_1\otimes\dots\otimes a_{m+1})=$$
\begin{align*}
&a_1f(a_2\otimes \dots\otimes a_{m+1})
+\sum_{j=1}^{m}(-1)^{j}f(a_1\otimes \dots\otimes a_{j-1}\otimes a_{j}a_{j+1}\otimes a_{j+2}\otimes \dots\otimes a_{m+1})\\
&+(-1)^{m+1}f(a_1\otimes \dots\otimes a_{m})a_{m+1},
\end{align*}
if $m\geq 1$ and $\partial^0(f)(a)=af-fa$ for any $f\in C^{0}(k\calC$). It is well known that $\HH^0(k\calC)=\Ker(\partial^0)=Z(k\calC)$ and $\HH^m(\calC):=\Ker (\partial^m)/\Ima (\partial^{m-1})$ for $m\geq 1$.

Since $\calC$ is finite, any $k$-linear map $f\in C^{m}(k\calC)$ is represented by a matrix $$(f_{g_1,\dots,g_m}^h)_{g_1,\dots,g_m,h\in\Mor(\calC)}$$ with respect to the basis $\Mor(\calC)$ of $k\calC$:
$$
f(g_1\otimes \dots\otimes g_m)=\sum_{h\in\Mor(\calC)}f_{g_1,\dots,g_m}^hh
$$
for all $g_1,\dots,g_m\in\Mor(\calC)$. In particular, the differential operator $\partial^m$ applied to $f$ is determined by 
$$\partial^{m}(f)(g_1\otimes \dots\otimes g_{m+1})=$$
\begin{equation}
  \label{HH_basis}
  =\sum_{h\in\Mor(\calC)}\left(f_{g_2,\dots,g_{m+1}}^hg_1h
+\sum_{j=1}^{m}(-1)^{j}f_{g_1,\dots,g_{j}g_{j+1},\dots,g_{m+1}}^{h}h
+(-1)^{m+1}f_{g_1,\dots,g_{m}}^hhg_{m+1}\right).
\end{equation}

\subsection{Simplicial cohomology}
\label{complex_nerve}
The simplicial cohomology $\Hh^{\bullet}(\calC,k)$ is the cohomology $\Hh^{\bullet}(B\calC,k)$ of the nerve of $\calC$ (see \cite[\S5]{Webb_representations_cohomology}). The nerve of $\calC$ is the simplicial set
$$
N\calC=\left\{
\sigma:
a_0\xrightarrow{\mathmakebox[1.5em]{g_0}}
a_1\xrightarrow{\mathmakebox[1.5em]{g_1}}
\dots\xrightarrow{\mathmakebox[1.5em]{g_{m-2}}}
a_{m-1}\xrightarrow{\mathmakebox[1.5em]{g_{m-1}}}
a_m
\,
\mid
\text{ $\sigma$ is an $m$-chain of morphisms in $\calC$}\right\}.
$$
The subset of $m$-chains is denoted by $N\calC_{m}$ and the $0$-chains are $NC_0=\Ob(\calC)$. The simplicial homology with coefficients in $k$ is obtained with the simplicial chain complex
$$
\dots
\xrightarrow{\mathmakebox[1.5em]{\delta_{m+1}}}
kN\calC_{m+1}
\xrightarrow{\mathmakebox[1.5em]{\delta_{m}}}
kN\calC_{m}
\xrightarrow{\mathmakebox[1.5em]{\delta_{m-1}}}
\dots
\xrightarrow{\mathmakebox[1.5em]{\delta_{1}}}
kN\calC_{1}
\xrightarrow{\mathmakebox[1.5em]{\delta_{0}}}
kN\calC_{0}
\xrightarrow{\mathmakebox[1.5em]{}} 0
$$
where $kN\calC_{m}$ is the $k$-vector space with basis $N\calC_{m}$ and where
$$
\delta_m=\sum_{i=0}^{m+1}(-1)^{i}d_i^{m}
$$
with $d_i^{m}:N\calC_{m+1}\rightarrow N\calC_{m}$ is the $i$-th face map (extended by linearity), i.e. for $1\leq i\leq m$:
\begin{align*}
&d_i^{m}(
a_0\xrightarrow{\mathmakebox[1.5em]{g_0}}a_1\xrightarrow{\mathmakebox[1.5em]{g_1}}\dots
a_{i-1}\xrightarrow{\mathmakebox[1.5em]{g_{i-1}}}a_{i}\xrightarrow{\mathmakebox[1.5em]{g_{i}}}a_{i+1}\dots
\xrightarrow{\mathmakebox[1.5em]{g_{m-1}}}a_{m}\xrightarrow{\mathmakebox[1.5em]{g_{m}}}a_{m+1}
)\\
&=a_0\xrightarrow{\mathmakebox[1.5em]{g_0}}a_1\xrightarrow{\mathmakebox[1.5em]{g_1}}\dots
a_{i-1}\xrightarrow{\mathmakebox[5.3em]{g_{i-1}\circ g_{i}}}a_{i+1}\dots
\xrightarrow{\mathmakebox[1.5em]{g_{m-1}}}a_{m}\xrightarrow{\mathmakebox[1.5em]{g_{m}}}a_{m+1};
\end{align*}
\begin{align*}
d_0^{m}(
a_0\xrightarrow{\mathmakebox[1.5em]{g_0}}a_1\xrightarrow{\mathmakebox[1.5em]{g_1}}\dots
\xrightarrow{\mathmakebox[1.5em]{g_{m-1}}}a_{m}\xrightarrow{\mathmakebox[1.5em]{g_{m}}}a_{m+1}
)
&
=
a_1\xrightarrow{\mathmakebox[1.5em]{g_1}}\dots
\xrightarrow{\mathmakebox[1.5em]{g_{m-1}}}a_{m}\xrightarrow{\mathmakebox[1.5em]{g_{m}}}a_{m+1},
\\
d_{m+1}^{m}(
a_0\xrightarrow{\mathmakebox[1.5em]{g_0}}a_1\xrightarrow{\mathmakebox[1.5em]{g_1}}\dots
\xrightarrow{\mathmakebox[1.5em]{g_{m-1}}}a_{m}\xrightarrow{\mathmakebox[1.5em]{g_{m}}}a_{m+1}
)
&
=
a_0\xrightarrow{\mathmakebox[1.5em]{g_0}}a_1\xrightarrow{\mathmakebox[1.5em]{g_1}}\dots
\xrightarrow{\mathmakebox[1.5em]{g_{m-1}}}a_{m}.
\end{align*}
One obtains the cochain complex for the calculation of the simplicial cohomology
$$
0\xrightarrow{\mathmakebox[1.5em]{}}
C^{0}(kN\calC,k)\xrightarrow{\mathmakebox[1.5em]{\delta^{0}}}
C^{1}(kN\calC,k)\xrightarrow{\mathmakebox[1.5em]{\delta^{1}}}
\dots\xrightarrow{\mathmakebox[1.5em]{\delta^{m-1}}}
C^{m}(kN\calC,k)\xrightarrow{\mathmakebox[1.5em]{\delta^{m}}}
C^{m+1}(kN\calC,k)\xrightarrow{\mathmakebox[1.5em]{\delta^{m+1}}}
\dots
$$
where $C^{m}(kN\calC,k)=\Hom_{k}(kN\calC_{m},k)$ and where $$\delta^{m}:\Hom_{k}(kN\calC_{m},k)\rightarrow \Hom_{k}(kN\calC_{m+1},k)$$ is given by $\delta^{m}(f)=f\circ\delta_{m}$. Explicitly, for $f\in \Hom_{k}(kN\calC_{m},k)$ and $$\sigma=
(a_0\xrightarrow{\mathmakebox[1.5em]{g_0}}
\dots
\xrightarrow{\mathmakebox[1.5em]{g_{m}}}
a_{m+1})
\in N\calC_{m+1}
$$
we have
$$
\delta^{m}(f)(
a_0\xrightarrow{\mathmakebox[1.5em]{g_0}}
a_1\xrightarrow{\mathmakebox[1.5em]{g_1}}
\dots\xrightarrow{\mathmakebox[1.5em]{g_{m-1}}}
a_{m}\xrightarrow{\mathmakebox[1.5em]{g_{m}}}
a_{m+1}
)
$$
$$
=f(a_1\xrightarrow{\mathmakebox[1.5em]{g_1}}\dots
\xrightarrow{\mathmakebox[1.5em]{g_{m}}}a_{m+1})
+
\sum_{i=1}^{m}(-1)^if(
a_0\xrightarrow{\mathmakebox[1.5em]{g_0}}
\dots
a_{i-1}\xrightarrow{\mathmakebox[1.5em]{g_{i}\circ g_{i-1}}}a_{i+1}\dots
\xrightarrow{\mathmakebox[1.5em]{g_{m}}}a_{m+1}
)
$$
$$
+
(-1)^{m+1}f(
a_0\xrightarrow{\mathmakebox[1.5em]{g_0}}
\dots
\xrightarrow{\mathmakebox[1.5em]{g_{m-1}}}a_{m}
).
$$

\subsection{Deterministic cancellative categories}
\label{sec:cancellative}
We begin with some properties of deterministic cancellative categories.
\begin{lem}
  \label{end_factorization}
  Let $x_1,x_2,x_3\in\Ob(\calC)$, $f,g\in\Hom_{\calC}(x_1,x_2)$ and $h\in\Hom_{\calC}(x_2,x_3)$.
  \begin{itemize}
 \item[(a)] If $\calC$ is a left deterministic and left cancellative category, then  there is unique $a\in\End_{\calC}(x_1)$ such that diagram (\ref{ad_morphism}) is commutative.
    
 \item[(b)] If $\calC$ is a right deterministic and right cancellative category, then  there is unique $b\in\End_{\calC}(x_2)$ such that diagram (\ref{ad_morphism}) is commutative.
 
 \item[(c)]
  Let $\calC$ be a deterministic cancellative category. For any $g\in\Mor_\calC(x_1,x_2)$ there is an isomorphism of monoids $\varphi_g:\End_{\calC}(x_1)\rightarrow \End_{\calC}(x_2)$ such that $g\circ a=\varphi_g(a)\circ g,$ with inverse $\varphi_g^{-1}$ satisfying $b\circ g=g\circ \varphi_g^{-1}(b)$ for any $a\in \End_{\calC}(x_1)$ and $b\in \End_{\calC}(x_2)$. Moreover we have $\varphi_{h\circ g}=\varphi_h\circ\varphi_g$.
 \end{itemize}
 \end{lem} 
\begin{proof}
  The claims a) and b) follow from the definition using the cancellative properties. For c), using statement b),  we define $\varphi_{g}(a)$ to be the unique element such that $g\circ a=\varphi_g(a)\circ g$. For two elements $a_1,a_2\in \End_{\calC}(x)$ we have
$$
\varphi_{g}(a_1\circ a_2)\circ g=g\circ(a_1\circ a_2)=(g\circ a_1)\circ a_2=\varphi_g(a_1)\circ g\circ a_2=\varphi_g(a_1)\circ\varphi(a_2)\circ g
$$
hence, since $\calC$ is right cancellative, we obtain $\varphi_{g}(a_1\circ a_2)=\varphi_g(a_1)\circ\varphi_g(a_2)$. Similarly we can construct the morphism of monoids $\varphi_g^{-1}:\End_{\calC}(x_2)\rightarrow \End_{\calC}(x_1)$. It is easy to verify that $\varphi_g^{-1}$ is an inverse of $\varphi_g$ and the last statement. 
  \end{proof}
   A left cancellative and rr-transitive category has the following characterization.
  \begin{prop}\label{prop2}
  Let $\calC$ be a left cancellative category.  A category $\calC$ is  rr-transitive if and only if for any $x_1,x_2\in\Ob(\calC)$ and any  $f,g\in\Hom_{\calC}(x_1,x_2)$ there is a unique $a\in\End_{\calC}(x_1)$ such that the following diagram commutes
\begin{equation}
  \label{triangle}
  \begin{tikzcd}
    x_1 \arrow[r,"g"] & x_2 \\
    x_1 \arrow[u,"a"]\arrow[ur,swap,"f"] 
  \end{tikzcd}. 
\end{equation}
In other words, for any $x_1,x_2\in\Ob(\calC), g\in\Hom_{\calC}(x_1,x_2)$  we have $\Hom_{\calC}(x_1,x_2)=g\circ \End_{\calC}(x_1)$.
  \end{prop}

  The category $F^{\ad}(\calC)$ can be viewed as a subcategory of the arrow category $\Arr(\calC)$, the category with objects $\Mor(\calC)$ and morphisms $(f,g):a\rightarrow b$ (with $a,b\in \Mor (\calC)$) for each commutative diagram of the form
$$
  \begin{tikzcd}
x_1 \arrow[r,"g"] & x_2 \\
y_1 \arrow[u,"a"]\arrow[r,"f"] & y_2\arrow[u,swap,"b"]
  \end{tikzcd}.
  $$
  Notice that one obtains $\Arr(\calC)$ by reversing one arrow in the definition of the Quillen's category of factorizations $F(\calC)$ \cite{Quillen}. A main ingredient for the proof in \cite{Xu_HHH} is the construction of a split surjective ring homomorphism  $\HH^{\bullet}(k\calC)\rightarrow\Hh^{\bullet}(BF(\calC),k)$. In this paper we replace $F(\calC)$ with $F^{\ad}(\calC)$ generalizing the explicit construction of the groupoid $Gr$ from \cite{Mishchenko}.

\begin{prop} Let $\calC$ be a cancellative category. The category $\calC$ is right deterministic and rr-transitive if and only if for any $a\in\End_{\calC}(x_1)$ and $g,f\in\Hom_{\calC}(x_1,x_2)$ there exists a unique $b\in\End_{\calC}(x_2)$ such that the following diagram commutes.
\begin{equation}
  \label{comm_diag_end}
  \begin{tikzcd}
x_1 \arrow[r,"g"] & x_2 \\
x_1 \arrow[u,"a"]\arrow[r,"f"] & x_2\arrow[u,swap,"b"]
  \end{tikzcd}.
\end{equation}
Similarly, the category $\calC$ is left deterministic and rr-transitive if and only if for any $b\in\End_{\calC}(x_2)$ and $g,f\in\Hom_{\calC}(x_1,x_2)$ there exists a unique $a\in\End_{\calC}(x_2)$ such that the above diagram commutes.
\end{prop}
\begin{proof}
  Choose $a\in\End_{\calC}(x_1)$ and $g,f\in\Hom_{\calC}(x_1,x_2)$. Since $\calC$ is rr-transitive and right cancellative, there is a unique $a'\in\End_{\calC}(x_1 )$ such that $g=f\circ a'$. Since $\calC$ is right deterministic and right cancellative, there is a unique $b$ such that $f\circ (a'\circ a)=b\circ f$. Thus we have a unique $b$ such that $g\circ a=b\circ f$. The other direction also follows. The second claim follows with a similar argument.
\end{proof}

\begin{prop}
  \label{star_groupoid}
  A  groupoid is cancellative, deterministic and rr-transitive.
\end{prop}

\begin{proof}
  Let $\calC$ be a groupoid. Since any morphism $f\in\Mor(\calC)$ has an inverse $f^{-1}\in\Mor(\calC)$, one easily checks that $\calC$ is cancellative. Consider the commutative diagram \eqref{comm_diag_end}. Since $g$ is invertible  the relation $g\circ a=b\circ g$ implies $b=g\circ a\circ g^{-1}$. A similar argument shows that $a$ is uniquely determined by $b$, hence $\calC$ is deterministic and rr-transitiv.
  \end{proof}

\begin{prop}
  \label{star_poset}
  A poset is cancellative, deterministic and rr-transitive.
\end{prop}

\begin{proof}
  Let $(\calC,\leq)$ be a poset and  let $g,h,f$ be morphisms such that $g\circ h$ and $g\circ f$ are morphisms. Then $s(h)\geq t(h)=s(g)\geq t(g)$ and $s(f)\geq t(f)=s(g)\geq t(g)$. If $g\circ h=g\circ f$ then $s(g\circ h)=s(h)=s(f)=s(g\circ f)$, hence $f$ and $g$ are the unique morphism between $s(h)=s(f)$ and $s(g)$. Thus $\calC$ is left cancellative. A similar argument shows that $\calC$ is also right cancellative. Consider the commutative diagram \eqref{comm_diag_end}. Since the only endomorphisms in $\calC$ are the identity morphisms, $a=\id_{x_1}$ and necessarily $b=\id_{x_2}$. Hence $\calC$ is deterministic and rr-transitive.
  \end{proof}
  
Next we give the promised example of a category satisfying the assumptions needed in our  main results, but which is not a groupoid or a poset.

\begin{ex}\label{ex6} Let $\calC$ be the category with two objects (denoted $x_1,x_2$), and morphisms
  $$
  
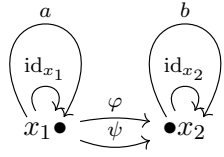
\begin{tikzcd}
    \arrow[out=120,in=60,loop,looseness=5, start anchor={[yshift=0ex]}, end anchor={[yshift=0ex]}]{}{\id_{x_{1}}}
    \arrow[out=120,in=60,loop,looseness=10, start anchor={[xshift=-.7ex,yshift=-.5ex]}, end anchor={[xshift=.7ex,yshift=-.5ex]}]{}[swap,yshift=1.3ex]{}{a} 
    \arrow[bend right=-10]{r}{\varphi}
    \arrow[bend right=20]{r}{\psi}
    x_{1}\bullet
    &
    \bullet x_2
    \arrow[out=120,in=60,loop,looseness=5, start anchor={[yshift=0ex]}, end anchor={[yshift=0ex]}]{}{\id_{x_2}}
    \arrow[out=120,in=60,loop,looseness=10, start anchor={[xshift=-.7ex,yshift=-.5ex]}, end anchor={[xshift=.7ex,yshift=-.5ex]}]{}[swap,yshift=1.3ex]{}{b}
  \end{tikzcd}
  $$
  given by $\Mor(\calC)=\{\id_{x_{1}},a,\id_{x_{2}},b,\varphi,\psi\}$ with the compositions of morphisms
  $$a\circ a=\id_{x_{1}},b\circ b=\id_{x_{2}},\varphi\circ a=\psi, \psi\circ a=\varphi, b\circ \varphi=\psi, b\circ \psi=\varphi$$
  
It is an easy exercise to verify that $\calC$ is deterministic, cancellative and rr-transitive finite category.
\end{ex}

\section{Graded Morphisms}
\label{Mishchenko_general}
\subsection{The graded morphism $\calT^{\bullet}$}
\label{epi_T}

For any $m\geq 0$ we define  $k$-linear maps
$$
\calT^m:C^{m}(k\calC)\rightarrow C^{m}(kNF^{\ad}(\calC),k)
$$
A $k$-linear map $X\in C^{m}(k\calC)$ is determined by its values on a basis of $(k\calC)^{\otimes m}$. If $m\geq 1$, for each $g_0,\dots,g_{m-1}\in\Mor(\calC)$ we have a unique decomposition
$$
X(g_{m-1}\otimes \dots\otimes g_{0})=\sum_{h\in\Mor(\calC)}X_{g_{m-1},\dots,g_{0}}^{h}h
$$
and we define $\calT^m(X)\in C^{m}(kNF^{\ad}(\calC),k)$ on the basis $NF^{\ad}(\calC)_m$ of $kNF^{\ad}(\calC)_m$ by
$$
\calT^m(X)\left(
\begin{tikzcd}
x_0 \arrow[r,"g_0"] & x_1 \arrow[r, "g_1"]& \cdots\arrow[r,"g_{m-2}"]  & x_{m-1}  \arrow[r,"g_{m-1}"] & x_m \\
x_0 \arrow[u,"a_0"]\arrow[r,"g_0"] & x_1\arrow[u,"a_1"]\arrow[r, "g_1"] & \cdots\arrow[r,"g_{m-2}"] & x_{m-1}\arrow[u,"a_{m-1}"]\arrow[r,"g_{m-1}"] & x_m \arrow[u,"a_m"]
  \end{tikzcd}
\right)
=X_{g_{m-1},\dots,g_{0}}^{g_{m-1}g_{m-2}\dots g_{2}g_{1}g_0a_0}.
$$
If $m=0$ then $C^{0}(k\calC)=k\calC$ and $X=\sum_{h\in\Mor(\calC)}X^{h}h$. In this case we let
$$
\calT^0(X)\left(
\begin{tikzcd}
x_0 \\
x_0 \arrow[u,"a_0"]
  \end{tikzcd}
\right)
=X^{a_0}.
$$
Extending $\calT^{m}(X)$ by linearity, we obtain a graded morphism $\calT^{\bullet}=(\calT^{m})_{m\geq 0}$ between the cochain complexes $C^{\bullet}(k\calC)$ and $C^{\bullet}(kNF^{\ad}(\calC),k)$.

\begin{prop}
  \label{complex_map}
  If $\calC$ is a finite cancellative category then $\calT^{m+1}\circ\partial^m=(-1)^{m+1}\delta^m\circ \calT^m$.
  \end{prop}
\begin{proof}
  We follow the argument in \cite[\S3.1]{Mishchenko} and show that for any $X\in C^{m}(k\calC)$ we have
  $$
  \calT^{m+1}(\partial^mX)(\sigma)=(-1)^{m+1}\delta^m(\calT^mX)(\sigma)
  $$
for all $(m+1)$-chains
$$
\sigma=
\left(\begin{tikzcd}
x_0 \arrow[r,"g_0"] & x_1 \arrow[r, "g_1"]& \cdots\arrow[r,"g_{m-1}"]  & x_{m}  \arrow[r,"g_{m}"] & x_{m+1} \\
x_0 \arrow[u,"a_0"]\arrow[r,"g_0"] & x_1\arrow[u,"a_1"]\arrow[r, "g_1"] & \cdots\arrow[r,"g_{m-1}"] & x_{m}\arrow[u,"a_{m}"]\arrow[r,"g_{m}"] & x_{m+1} \arrow[u,"a_{m+1}"]
\end{tikzcd}\right).
$$
The left-hand side $\calT^{m+1}(\partial^mX)(\sigma)$ is the coefficient of $h'=g_{m}g_{m-1}\dots g_{1}g_{0}a_{0}$ in 
$(\partial^mX)(g_m\otimes \dots\otimes g_0)$ which by \eqref{HH_basis} equals
\begin{align*}
  \sum_{h\in\Mor(\calC)}\left(X_{g_{m-1},\dots,g_0}^{h}g_mh
+\sum_{j=0}^{m-1}(-1)^{m-j}X_{g_m,\dots,g_{j+2},g_{j+1}g_{j},g_{j-1},\dots,g_{0}}^{h}h
+ (-1)^{m+1}X^{h}_{g_m,\dots,g_{1}}hg_0\right).
\end{align*}
Since $\Mor(\calC)$ is a basis for $k\calC$, the coefficient of $h'$ is
$$
  \sum_{h\in\Mor(\calC)}^{g_mh=h'}X_{g_{m-1},\dots,g_{0}}^{h}
+ \sum_{j=0}^{m-1}(-1)^{m-j}X_{g_m,\dots,g_{j+2},g_{j+1}g_{j},g_{j-1},\dots,g_{0}}^{h'}
+ (-1)^{m+1}\sum_{h\in\Mor(\calC)}^{hg_0=h'}X^{h}_{g_m,\dots,g_{1}}.
$$
From the commutativity of the diagram $\sigma$, the equation $hg_0=h'$ admits the solution $h_1=g_m\dots g_1a_1$. If $h_2$ is another solution to $hg_0=h'$ then $h_1g_0=h_2g_0$ and since $\calC$ is right cancellative, we have $h_1=h_2$. For the equation $g_mh=h'$, since $\calC$ is left cancellative, we have the unique solution $h=g_{m-1}\dots g_{0}a_{0}$. Hence $\calT^{m+1}(\partial^mX)(\sigma)$ equals
$$
\underbrace{X_{g_{m-1},\dots,g_{0}}^{g_{m-1}\dots g_{0}a_{0}}}_{\calT^m(X)\circ d^m_{m+1}(\sigma)}
+
\underbrace{\sum_{j=1}^{m}(-1)^{m-j+1}X_{g_{m},\dots,g_{j+1},g_{j}g_{j-1},g_{j-2}\dots,g_{0}}^{g_{m} g_{m-1} \dots g_{1}g_{0}a_0}}_{(-1)^{m+1}\sum_{j=1}^{m}(-1)^{j}\calT^m(X)\circ d^m_{j}(\sigma)}
+
(-1)^{m+1}
\underbrace{X^{g_{m}\dots g_{1}a_1}_{g_{m},\dots,g_{1}}}_{\calT^m(X)\circ d^m_{0}(\sigma)}.
$$
From the description of $\delta^{m}$ in Subsection \ref{complex_nerve} this also equals $(-1)^{m+1}\delta^m(\calT^m(X))(\sigma)$. This concludes the proof.

For clarity we treat the case $m=0$ separately. If $X\in C^{0}(k\calC)$ then $\calT^{1}(\partial^{0}X)(\sigma)$ is the coefficient of $g_0a_0$ in $(\partial^{0}X)(g_0)=g_0X+Xg_0=\sum_{h\in\Mor(\calC)}X^{h}g_0h-X^{h}hg_0$. If $g_0h=g_0a_0$ then $h=a_0$ and if $hg_0=g_0a_0$ then $h=a_1$. Hence $\calT^{1}(\partial^{0}X)(\sigma)=X^{a_0}-X^{a_1}$. On the other hand $X^{a_0}=\calT^{0}(X)\circ d^{0}_{1}$ and $X^{a_1}=\calT^{0}(X)\circ d^{0}_{0}$. Thus $\calT^{1}(\partial^{0}X)(\sigma)=-\delta^{0}(\calT^{m}X)(\sigma)$.
\end{proof}

\subsection{A section $\calX^{\bullet}$ for $\calT^{\bullet}$}
\label{section_X}
In this subsection we assume that $\calC$ is right deterministic and right cancellative. With these assumptions we can  describe a section $\calX^{\bullet}:C^{\bullet}(kNF^{\ad}(\calC),k)\rightarrow C^{\bullet}(k\calC)$ for the map $\calT^{\bullet}$. The $k$-linear map $T\in C^{m}(kNF^{\ad}(\calC),k)$ is determined by its values on  $NF^{\ad}(\calC)_m$.
For each
\begin{equation}
  \label{my_sigma}
\sigma
=
\left(\begin{tikzcd}
x_0 \arrow[r,"g_0"] & x_1 \arrow[r, "g_1"]& \cdots\arrow[r,"g_{m-2}"]  & x_{m-1}  \arrow[r,"g_{m-1}"] & x_m \\
x_0 \arrow[u,"a_0"]\arrow[r,"g_0"] & x_1\arrow[u,"a_1"]\arrow[r, "g_1"] & \cdots \arrow[r,"g_{m-2}"]& x_{m-1}\arrow[u,"a_{m-1}"]\arrow[r,"g_{m-1}"] & x_m \arrow[u,"a_m"]
\end{tikzcd}\right)
\in NF^{\ad}(\calC)_m,
\end{equation}
by Lemma \ref{end_factorization} (b), any sequence $g_{m-1},\dots,g_{0},a_0$ of composable morphisms in $\calC$ (like above) determines uniquely  the morphisms $a_1,\ldots, a_m$. Hence, we shall adopt the notation 
$$
T_{g_{m-1},\dots,g_{0}}^{g_{m-1}g_{m-2}\cdots g_{1} g_{0}a_{0}}:=T(\sigma)
$$
if $m\geq 1$ and $T^{a_0}:=T(\sigma)$ if $m=0$. If $m\geq 1$, we define $\calX^m:C^{m}(kNF^{\ad}(\calC),k)\rightarrow C^{m}(k\calC)$
by 
$$
\calX^m(T)(g_{m-1}\otimes \dots\otimes g_{0})
=\sum_{a_0\in\End_{\calC}(s(g_0))}T_{g_{m-1},\dots,g_{0}}^{g_{m-1}g_{m-2}\dots g_{1} g_{0}a_0}
g_{m-1}g_{m-2}\cdots g_{2} g_{1} g_{0}a_{0}.
$$
for each $g_0,\dots,g_{m-1}\in\Mor(\calC)$. If $m=0$ we let $\calX^0(T)=\sum_{a_0\in\End(\calC)}T^{a_0}a_0$. For all $m\geq 0$ we extend $\calX^m(T)$ by linearity. If the composition $g_{m-1}\circ g_{m-2}\cdots g_{2}\circ g_{1}\circ g_{0}\circ a_{0}$ does not exist, the element $T_{g_{m-1},\dots,g_{0}}^{g_{m-1}g_{m-2}\dots g_{1} g_{0}a_0}$ is not defined, however, in this cases $g_{m-1}g_{m-2}\cdots g_{2} g_{1} g_{0}a_{0}$ is zero in $k\calC$, i.e. the value of $\calX^m(T)$ on non-composable morphisms is zero.

\begin{prop}
  \label{complex_map_X}
  If $\calC$ is a finite deterministic cancellative category  then $\calX^{m+1}\circ \delta^m=(-1)^{m+1}\partial^m\circ \calX^{m}$.
\end{prop}
\begin{proof}
  For $m\geq 0$, we show that
  \begin{equation}
    \label{complex_map_X_eq}
  (-1)^{m+1}
  \calX^{m+1}(\delta^mT)(g_{m}\otimes \dots\otimes g_0)
  =
  \partial^m(\calX^{m}T)(g_{m}\otimes \dots\otimes g_0)
  \end{equation}
  for any $T\in C^{m}(kNF^{\ad}(\calC),k)$ and for all $g_0,\dots,g_{m}\in \Mor(\calC)$. The right-hand side is
  $$
  g_m (\calX^{m}T)(g_{m-1}\otimes \dots\otimes g_0)+\sum_{j=0}^{m-1}(-1)^{m-j}(\calX^{m}T)(\dots\otimes g_{j+1}g_{j}\otimes \dots)+(-1)^{m+1}(\calX^{m}T)(g_{m}\otimes \dots\otimes g_1)g_0.
  $$
  By definition of $\calX^{m}$, if $g_{m-1}\circ\dots\circ g_0$ exists, then $(\calX^{m}T)(g_{m-1}\otimes \dots\otimes g_0)\in\Mor_{\calC}(s(g_0),t(g_{m-1}))$, hence if $g_{m}\circ\dots\circ g_0$ does not exist, then $t(g_{m-1})\neq s(g_m)$ and thus $g_m (\calX^{m}T)(g_{m-1}\otimes \dots\otimes g_0)=0$. A similar argument shows that if $g_{m}\circ\dots\circ g_0$ does not exist then all other terms in the above sum are zero. In this case the left-hand side of \eqref{complex_map_X_eq} is also zero. Assume that $g_{m}\circ\dots\circ g_0$ exists. By definition of $\calX^{m}$ the right-hand side of \eqref{complex_map_X_eq} further equals
  $$
  g_m \sum_{a_0\in\End_{\calC}(s(g_0))}T_{g_{m-1},\dots,g_0}^{g_{m-1}\cdots g_0a_0}g_{m-1}\cdots g_0a_0
  +\sum_{j=0}^{m-1}(-1)^{m-j}
  \sum_{a_0\in\End_{\calC}(s(g_0))}T_{g_m,\dots,g_{j+1}g_{j},\dots g_0}^{g_{m}\cdots g_0a_0}g_{m}\cdots g_0a_0
  $$
  $$
  +(-1)^{m+1}\sum_{a_1\in\End_{\calC}(s(g_1))}T_{g_{m},\dots,g_1}^{g_{m}\cdots g_1a_1}g_{m}\cdots g_1a_1g_0.
  $$
  Since $\calC$ is deterministic, by Lemma \ref{end_factorization} (c), we have an isomorphism of monoids $\varphi_{g_{0}}:\End_{\calC}(s(g_0))\rightarrow \End_{\calC}(s(g_1))$ such that $\varphi_{g_0}(a_0)\circ g_0=g_0\circ a_0$ for all $a_0\in \End_{\calC}(s(g_0))$. Hence
$$
\sum_{a_1\in\End_{\calC}(s(g_1))}T_{g_{m},\dots,g_1}^{g_{m}\cdots g_1a_1}g_{m}\cdots g_1a_1g_0
=
\sum_{a_0\in\End_{\calC}(s(g_0))}T_{g_{m},\dots,g_1}^{g_{m}\cdots g_1\varphi_{g_0}(a_0)}g_{m}\cdots g_1g_0a_0
$$
which equals $\calX^{m+1}(T\circ d^{m}_{0})(g_m\otimes\dots\otimes g_0)$. We also have
$$
\sum_{a_0\in\End_{\calC}(s(g_0))}
T_{g_{m-1},\dots,g_0}^{g_{m-1}\cdots g_0a_0}g_{m}\cdots g_0a_0
=
\calX^{m+1}(T\circ d^{m}_{m+1})(g_m\otimes\dots\otimes g_0)
$$
and for $0\leq j\leq m-1$ we have
$$
\sum_{a_0\in\End_{\calC}(s(g_0))}
T_{g_m,\dots,g_{j+1}g_{j},\dots,g_0}^{g_{m}\cdots g_0a_0}g_{m}\cdots g_0a_0
=
\calX^{m+1}(T\circ d^{m}_{m-j})(g_m\otimes\dots\otimes g_0).
$$
Since
$$
\sum_{j=0}^{m-1}
(-1)^{m-j}\calX^{m+1}(T\circ d^{m}_{m-j})
=
(-1)^{m-1}
\sum_{i=1}^{m}
(-1)^{i}\calX^{m+1}(T\circ d^{m}_{i})
$$
the claim follows from the description of $\delta^{m}$ in Subsection \ref{complex_nerve}.

For clarity we treat the case $m=0$ separately. If $T\in C^{0}(kNF^{\ad}(\calC),k)$ then by definition we have $\calX^0(T)=\sum_{a_0\in\End(\calC)}T^{a_0}a_0$. Thus $$(\partial^{0}\circ \calX^0)(T)(g_0)=\sum_{a_0\in\End(\calC)}T^{a_0}(g_0a_0-a_0g_0).$$ Hence $(\partial^{0}\circ \calX^0)(T)(g_0)$ equals
$$
\sum_{a_0\in\End_{\calC}(s(g_0))}T^{a_0}g_0a_0-\sum_{a_1\in\End_{\calC}(t(g_0))}T^{a_1}a_1g_0
=\sum_{a_0\in\End_{\calC}(s(g_0))}T^{a_0}g_0a_0-\sum_{a_0\in\End_{\calC}(s(g_0))}T^{\varphi_{g_0}(a_0)}\varphi_{g_0}(a_0)g_0
$$
for the isomorphism $\varphi_{g_0}:\End_{\calC}(s(g_0))\rightarrow \End_{\calC}(t(g_0))$ given in Lemma \ref{end_factorization} (c). On the other hand, $$\calX^{1}(\delta^{0}T)(g_0)=\sum_{a_0\in\End_{\calC}(s(g_0))}(\delta^{0}T)_{g_0}^{g_0a_0}g_0a_0$$ for any $g_0\in k\calC$, where by definition $(\delta^{0}T)_{g_0}^{g_0a_0}=(\delta^{0}T)(\sigma)$ with $\sigma$ as in \eqref{my_sigma} with $m=1$. Since $(\delta^{0}T)(\sigma)=T^{a_1}-T^{a_0}$ where $T^{a_1}=T\circ d^{0}_{0}(\sigma)$ and $T^{a_0}=T\circ d^{0}_{1}(\sigma)$ it follows that $\calX^{1}(\delta^{0}T)(g_0)$ equals
$$
  \sum_{a_0\in\End_{\calC}(s(g_0))}(T^{\varphi_{g_0}(a_0)}-T^{a_0})g_0a_0
=\sum_{a_0\in\End_{\calC}(s(g_0))}T^{\varphi_{g_0}(a_0)}g_0a_0
-\sum_{a_0\in\End_{\calC}(s(g_0))}T^{a_0}g_0a_0.
$$
Since $g_0a_0=\varphi(a_0)g_0$, we obtain $\partial^{0}\circ \calX^0(T)(g_0)=-\calX^{1}(\delta^{0}T)(g_0)$.
\end{proof}

\begin{prop}
  \label{prop_X_section}
  Let $\calC$ be a right deterministic cancellative category. The map $\calX^{\bullet}$ is a $k$-linear right inverse for $\calT^{\bullet}$.
\end{prop}

\begin{proof}
  We show that $$(\calT^{m}\circ\calX^{m})(T)=T$$ for all $T\in C^{m}(kNF^{\ad}(\calC),k)$ and for all $m\geq 0$. For each $\sigma$ as in \eqref{my_sigma}, $\calT^{m}(\calX^{m}(T))(\sigma)$ is the coefficient of $g_{m-1}g_{m-2}\dots g_{2}g_{1}g_0a_{0}$ in $\calX^{m}(T)(g_{m-1},\dots,g_{0})$. Since $\calC$ is left cancellative, the equation $$g_{m-1}g_{m-2}\cdots g_{1} g_{0}x=g_{m-1}g_{m-2}\cdots g_{1} g_{0}a_{0}$$ has $x=a_0$ as unique solution. If $m=0$ the above equation is $x=a_0$ and we don't need $\calC$ to be left cancellative in this case. Hence $$\calT^{m}(\calX^{m}(T))(\sigma)=T_{g_{m-1},\dots,g_{0}}^{g_{m-1}g_{m-2}\dots g_{2}g_{1}g_0a_{0}}=T(\sigma)$$ and the claim follows. 
\end{proof}

\section{Relative Hochschild cohomology}
\label{sec:relative_HH}

The notion of relative Hochschild cohomology introduced in \cite{Gerstenhaber_Schack_85} simplifies the comparison of Hochschild cohomology and simplicial cohomology done in \cite{Gerstenhaber_Schack_83}. By \cite[Theorem 1.2]{Gerstenhaber_Schack_85}, if $R$ is a separable  subalgebra \cite[Definition 4.1.7]{With_book} of $k\calC$, then the $R$-relative Hochschild cohomology of $k\calC$ is isomorphic to the Hochschild cohomology of $k\calC$. For our purpose, we choose $R$ to be the subalgebra $k\calC^{\id}$ of $k\calC$ generated by the set of all identity morphisms on each object.

\begin{lem}
  \label{relative_cochains}
  The algebra $k\calC^{\id}$ is separable.
\end{lem}
\begin{proof}
  The algebra $R=k\calC^{\id}$ is separable if and only if there is an idempotent $e=\sum_{i}u_i\otimes v_i$ in the enveloping algebra $R\otimes R^{\op}$ of $R$ such that $\sum_iu_iv_i=1$ and $(r\otimes 1)e=(1\otimes r)e$. One checks that $e=\sum_{x\in\Ob(\calC)}1_x\otimes 1_x$ satisfies this property.
  \end{proof}

   As in \cite[\S1]{Gerstenhaber_Schack_85} the $k\calC^{\id}$-relative Hochschild cohomology of $k\calC$ is the cohomology of the subcomplex of $k\calC^{\id}$-relative cochains $C^{\bullet}(k\calC,k\calC^{\id})\subseteq C^{\bullet}(k\calC)$. The $k\calC^{\id}$-relative $m$-cochain group $C^{m}(k\calC,k\calC^{\id})$ consists of $k$-linear functions $f$ satisfying
\begin{itemize}
\item $f(ag_1\otimes \dots\otimes g_{m})=af(g_1\otimes \dots\otimes g_{m})$,
\item $f(g_1\otimes \dots\otimes g_{m}a)=f(g_1\otimes \dots\otimes g_{m})a$,
  \item $f(g_1\otimes\dots\otimes g_ia\otimes g_{i+1}\otimes \dots \otimes g_m)=f(g_1\otimes\dots\otimes g_i\otimes ag_{i+1}\otimes \dots\otimes g_m)\quad 1\leq i\leq m-1,$
\end{itemize}
for all $g_1,\dots,g_{m}\in\Mor(\calC)$ and any $a\in k\calC^{\id}$. Notice that $C^{0}(k\calC,k\calC^{\id})=C^{0}(k\calC)=k\calC$.

\begin{prop}
  \label{characterization_strict_cochains}
  If $m\geq 1$, an element $f\in C^{m}(k\calC)$ lies in $C^{m}(k\calC,k\calC^{\id})$ if and only if for any $g_1,\dots,g_{m}\in\Mor(\calC)$ the following holds
  \begin{itemize}
  \item $f(g_1\otimes \dots\otimes g_{m})=0$ if $t(g_{i+1})\neq s(g_{i}),$ for some $1\leq i\leq m-1$ 
  
  and
  
  \item $f(g_1\otimes \dots\otimes g_{m})\in k\left[\Hom_{\calC}(s(g_{m}),t(g_{1}))\right]$, otherwise.
  \end{itemize}
\end{prop}

\begin{proof}
  Let $f$ be a $k\calC^{\id}$-relative $m$-cochain. For any $g_1,\dots,g_{m}\in\Mor(\calC)$ we have
  $$
  f(g_1\otimes \dots\otimes g_{m})
  =f(\id_{t(g_1)}g_1\otimes \dots\otimes g_{m})
  =\id_{t(g_1)}f(g_1\otimes \dots\otimes g_{m}).
  $$
  Since $f(g_1\otimes \dots\otimes g_{m})=\sum_{h\in\Mor(\calC)}f_{g_1,\dots,g_{m}}^{h}h$, it follows that the non-zero terms in this sum are $f_{g_1\otimes \dots\otimes g_{m}}^{h}h$ with $t(h)=t(g_1)$. A similar argument shows that $s(h)=s(g_{m})$, hence
  $$
  f(g_1\otimes \dots\otimes g_{m})=\sum_{h\in\Hom_{\calC}(s(g_{m}),t(g_{1}))}f_{g_1,\dots,g_{m}}^{h}h.
  $$
  Let $i\in\{1,\dots,m-1\}$. Moreover 
$$
    f(g_1\otimes\dots\otimes g_i\otimes g_{i+1}\otimes \dots \otimes g_m)
    =f(g_1\otimes\dots\otimes g_i\otimes \id_{t(g_{i+1})}g_{i+1}\otimes \dots \otimes g_m)$$
    
    $$=f(g_1\otimes\dots\otimes g_i\id_{t(g_{i+1})}\otimes g_{i+1}\otimes \dots \otimes g_m).
$$
    Thus, if $s(g_i)\neq t(g_{i+1})$ then $g_i\id_{t(g_{i+1})}=0$, in which case $f(g_1\otimes\dots\otimes g_i\otimes g_{i+1}\dots \otimes g_m)=0$. Hence a $k\calC^{\id}$-relative $m$-cochain has the indicated properties. Conversely, if $f$ has the indicated properties, one checks that it is $k\calC^{\id}$-relative using the fact that $k\calC^{\id}$ is generated by $\{\id_{x}:x\in\Ob(\calC)\}$.
\end{proof}

\begin{cor}
  \label{deterministic_cochains}
  Let $\calC$ be a left cancellative rr-transitive category. Then $f\in C^m(k\calC,k\calC^{\id})$ if and only if
  $$
  f(g_{m-1}\otimes \dots\otimes g_{0})=\sum_{a\in\End_{\calC}(s(g_0))}f_{g_{m-1}, \dots, g_{0}}^{g_{m-1}g_{m-2}\dots g_{1}g_{0}a}g_{m-1}g_{m-2}\dots g_{1}g_{0}a.
  $$
  for any $g_0,\dots,g_{m-1}\in\Mor(\calC)$.
\end{cor}
\begin{proof}
  By Proposition \ref{characterization_strict_cochains}, for any $f\in C^{m}(k\calC,k\calC^{\id})$ and any $g_0,\dots,g_{m-1}\in\Mor(\calC)$ we have
  $$
  f(g_{m-1}\otimes \dots\otimes g_{0})=\sum_{h\in\Hom_{\calC}(s(g_0),t(g_{m-1}))}f_{g_{m-1},\dots,g_{0}}^{h}h
  $$
  where the sum is zero if $t(g_i)\neq s(g_{i+1}),$ for some $0\leq i\leq m-2$. Since $\calC$ is rr-transitive and left cancellative, the claim follows by Proposition \ref{prop2} with $g:=g_{m-1}\circ g_{m-2}\circ \dots \circ g_{1}\circ g_{0}$ and $f:=h$.
\end{proof}

\begin{proof}[Proof of Theorem \ref{HHH_epi}]
 For a small category $\calC$ we have (see \cite[\S5]{Webb_representations_cohomology})
  $$
  \Hh^{\bullet}(F^{\ad}(\calC),k)\cong\Hh^{\bullet}(BF^{\ad}(\calC))\cong \Hh^{\bullet}(C^{\bullet}(kNF^{\ad}(\calC),k),\delta^{\bullet}).
  $$
If we modify the differential by a non-zero constant, $\alpha^{m}=(-1)^{m+1}\delta^m$, for any non-negative integer $m$, the cohomology doesn't change, $\Hh^{\bullet}(F^{\ad}(\calC),k)\cong\Hh^{\bullet}(C^{\bullet}(kNF^{\ad}(\calC),k),\alpha^{\bullet})$, but now $\calT^{\bullet}$ becomes a cochain map if $\calC$ is cancellative (by Proposition \ref{complex_map}) and $\calX^{\bullet}$ become a cochain map if $\calC$ is also deterministic (by Proposition \ref{complex_map_X}). With Proposition \ref{prop_X_section}, the map $\calT^{\bullet}$ is a surjective cochain map if $\calC$ is deterministic and cancellative, hence it induces a surjective homomorphism in cohomology.

Since $k\calC^{\id}$ is a separable subalgebra of $k\calC$ (by Lemma \ref{relative_cochains}), as in \cite[\S1]{Gerstenhaber_Schack_85},  $\HH^{\bullet}(k\calC)$ is isomorphic to the $k\calC^{\id}$-relative Hochschild cohomology. By \cite[Theorem 1.2]{Gerstenhaber_Schack_85} the inclusion of complexes $C^{\bullet}(k\calC,k\calC^{\id})\rightarrow C^{\bullet}(k\calC)$ induces an isomorphism of cohomologies.  That is, we can replace the cohomology of the complex $C^{\bullet}(k\calC)$ by  the cohomology of its  subcomplex $C^{\bullet}(k\calC,k\calC^{\id})$, consisting of $k\calC^{\id}$-relative cochains. 

Assume next that $\calC$ is rr-transitive. Note that by definition of $\calX^{\bullet}$ and Corollary \ref{deterministic_cochains} we have $\im(\calX^{\bullet})\subset C^{\bullet}(k\calC,k\calC^{\id})$. As in the proof of Proposition \ref{prop_X_section} one checks that $\calX^{\bullet}$ is a $k$-linear right inverse of the restriction $\check \calT^{\bullet}$ of $\calT^{\bullet}$ to  $C^{\bullet}(k\calC,k\calC^{\id})$. 

Moreover, for $X\in  C^m(k\calC,k\calC^{\id})$ and $g_0,\ldots,g_{m-1}\in\Mor(\calC)$ we obtain
\begin{align*}
\calX^m(\check \calT^{m}(X))(g_{m-1}\otimes \ldots\otimes g_{0})&=\sum_{a\in\End_{\calC}(s(g_0))}(\calT^{m}(X))_{g_{m-1},\ldots,g_0}^{g_{m-1}\ldots g_0a}\ g_{m-1}\ldots g_0a\\
&=\sum_{a\in\End_{\calC}(s(g_0))}X_{g_{m-1}\ldots,g_0}^{g_{m-1}\ldots g_0a}\ g_{m-1}\ldots g_0a\\&=
X(g_{m-1}\otimes \ldots\otimes g_0),
\end{align*}
where the last equality is true by Corollary \ref{deterministic_cochains}. Hence $\calX^{\bullet}$ is also a left inverse for $\check\calT^{\bullet}$.
\end{proof}

\section{Derivations on $k\calC$ and characters on $F^{\ad}(\calC)$}
\label{sec:ders_chars}

\subsection{Graded derivations}\label{subsec:gradder}

Any category $\calC$ gives rise to a semigroup with zero which we denote by $S$. The set $S$ is 
$$
S:=\{(x_1,x_2)\in\Ob(\calC)\times\Ob(\calC)| \Hom_{\calC}(x_1,x_2)\neq\emptyset\}\cup\{(0,0)\}
$$
and the operation in $S$ is 
$$(x_1,x_2)(x_3,x_4)=\left\{
\begin{array}{ll}
  (x_1,x_4) &\text{ if $x_2=x_3$}\\
  0, &\text{ if $x_2\neq x_3$}
  \end{array}
\right.$$ 
The category algebra $k\calC$ is graded by $S$ with $k\calC=\bigoplus_{(x,y)\in S} R_{(x_1,x_2)}$ where $R_{(x_1,x_2)}=k[\Hom_{\calC}(x_1,x_2) ]$ if $(x_1,x_2)\in S\setminus \{(0,0)\}$ and $R_{(0,0)}=0$.

A linear endomorphism $\phi\in\End_k(k\calC)$ is called $S$-\emph{graded} if $\phi$ preserves the grading by $S$, i.e. $\phi(R_{(x_1,x_2)})\subseteq R_{(x_1,x_2)}$ for any $(x_1,x_2)\in S$.
 We denote by $\End^{\gr}(k\calC)$ the $k$-subspace of all $S$-graded endomorphisms and by $\Der^{\gr}(k\calC)$ the $k$-subspace of all $S$-graded derivations of $k\calC$. Since $C^{1}(k\calC)=\End_k(k\calC)$, it follows from Proposition \ref{characterization_strict_cochains} that $\End^{\gr}(k\calC)=C^{1}(k\calC,k\calC^{\id})$, i.e. $S$-graded endomorphisms are $k\calC^{\id}$-relative $1$-cochain maps.

\subsection{Characters on $F^{\ad}(\calC)$} One can extend the notion of \emph{character}, used in \cite[\S3]{Arutyunov_Mishchenko} for the group case, to any finite category $\calD$ as follows. Denote by $k^{\Mor(\calD)}$ the $k$-vector space of all set maps from $\Mor(\calD)$ to $k$. A set map $T\in k^{\Mor(\calD)}$ is called a \emph{character} if
  $$
  T(\eta\circ\zeta)=T(\eta)+T(\zeta)
  $$
  for all morphisms $\zeta,\eta\in\Mor(\calD)$ such that $s(\eta)=t(\zeta)$. We denote the $k$-subspace in $k^{\Mor(\calD)}$ consisting of all characters on $\calD$ by $\Char(\calD)$.

  Here we are interested in the case where $\calD=F^{\ad}(\calC)$ and we identify  $C^{1}(kNF^{\ad}(\calC),k)$ with $k^{\Mor(F^{\ad}(\calC))}$. Recall that for two  morphisms $\zeta,\eta$ in $F^{\ad}(\calC)$ such that $s(\eta)=t(\zeta)$ we have:
\begin{equation}
  \label{zeta_eta}
  \zeta=
  \left(\begin{tikzcd}
    x_1 \arrow[r,"g"] & x_2 \\
    x_1 \arrow[u,"a"]\arrow[r,"g"] & x_2\arrow[u,"b"]
  \end{tikzcd}\right)
  \quad\eta=
  \left(\begin{tikzcd}
    x_2 \arrow[r,"f"] & x_3 \\
    x_2 \arrow[u,"b"]\arrow[r,"f"] & x_3\arrow[u,"c"]
  \end{tikzcd}\right)
  \quad\eta\circ\zeta=
  \left(\begin{tikzcd}
    x_1 \arrow[r,"f\circ g"] & x_3 \\
    x_1 \arrow[u,"a"]\arrow[r,"f\circ g"] & x_3\arrow[u,"c"]
  \end{tikzcd}\right).
\end{equation}
In Section \ref{Mishchenko_general} we defined the map $\calT^1:\End_k(k\calC)\rightarrow k^{\Mor(F^{\ad}(\calC))}$. For any $X\in\End_k(k\calC)$ we have a unique decomposition $X(g)=\sum_{h\in\Mor(\calC)}X^{h}_{g}h$ for all $g\in \Mor(\calC)$ and $\calT^{1}(X)$ is defined by $\calT^{1}(X)(\sigma)=X^{g\circ a}_g$ for all $1$-simplices $\sigma$ as in \eqref{ad_morphism}. If  $\calC$ is right deterministic and right cancellative we defined $\calX^1:k^{\Mor(F^{\ad}(\calC))}\rightarrow \End_k(k\calC)$ by
  $$
  T\mapsto\calX^1(T),\quad \calX^1(T)(g)=\sum_{a\in\End_{\calC}(s(g))}T^{g\circ a}_{g}g\circ a
  $$
  for all $g\in k\calC$, where $T^{g\circ a}_{g}$ is the value $T(\sigma)$ with $\sigma$ as in \eqref{ad_morphism}.

  \begin{proof}[Proof of Theorem \ref{HH1H1}]
    We adopt the notations $\calT:=\calT^1|_{\Der^{\gr}(k\calC)}$ for the restriction of $\calT^1$ to $\Der^{\gr}(k\calC)$ and $\calX:=\calX^1|_{\Char(F^{\ad}(\calC))}$ for the restriction of $\calX^1$ to $\Char(F^{\ad}(\calC))$.
    
    First we show that $\im(\calT)\subseteq\Char(F^{\ad}(\calC))$, i.e. we show that for $X\in\Der^{\gr}(k\calC)$ and any composable morphisms $\zeta,\eta\in\Mor(F^{\ad}(\calC))$ as in \eqref{zeta_eta} we have
    $$
    \calT(X)(\eta\circ\zeta)=\calT(X)(\zeta)+\calT(X)(\eta).
    $$
    By definition of $\calT$, of $\zeta$ and of $\eta$, this is equivalent to
    $$
    X^{f\circ g\circ a}_{f\circ g}=X^{g\circ a}_{g} +X^{f\circ b}_{f}.
    $$
    Since $\calC$ is rr-transitive and left cancellative, by Corollary \ref{deterministic_cochains}, for any $h\in\Mor(\calC)$ we have
    $$
    X(h)=\sum_{d\in\End_{\calC}(s(h))}X_{h}^{h\circ d}h\circ d.
    $$
    Since $X$ is a derivation we have $X(fg)=X(f)g+fX(g)$ which is equivalent to
    $$
    \sum_{d\in\End_{\calC}(s(g))}X_{f\circ g}^{f\circ g\circ d}f\circ g\circ d
    =
    \sum_{d'\in\End_{\calC}(s(f))}X_{f}^{f\circ d'}f\circ d'\circ g
    +
    \sum_{d\in\End_{\calC}(s(g))}X_{g}^{g\circ d}f\circ g\circ d
    .
    $$
    Since $\calC$ is deterministic and cancellative, by Lemma \ref{end_factorization} (c), there is an isomorphism of monoids $\varphi_g^{-1}:\End_{\calC}(t(g))\rightarrow\End_{\calC}(s(g))$ such that $d'\circ g=g\circ \varphi_{g}^{-1}(d')$. Hence,
        $$
    \sum_{d\in\End_{\calC}(s(g))}X_{f\circ g}^{f\circ g\circ d}f\circ g\circ d
    =
    \sum_{d'\in\End_{\calC}(s(f))}
    X_{f}^{f\circ d'} f\circ g\circ\varphi_g^{-1}(d')
    +\sum_{d\in\End_{\calC}(s(g))}X_{g}^{g\circ d}f\circ g\circ d.
    $$
    Since $\calC$ is cancellative, we have $f\circ g\circ d=f\circ g\circ a$ if and only if $d=a$ and,  $\varphi_g^{-1}(d')=a$ if and only if $d'=b$. Hence the coefficient of $f\circ g\circ a$ is $$X_{f\circ g}^{f\circ g\circ a}=X_{f}^{f\circ b}+X_{g}^{g\circ a}.$$ This shows that $\calT$ maps graded derivations to characters.
    
By definition of $\calX$, it is clear that $\im(\calX)\subseteq \End^{\gr}(k\calC)$. Thus, it is enough to show that $\im(\calX)\subseteq\Der(k\calC)$, i.e. we show that for $T\in\Char(F^{\ad}(\calC))$ and any $f,g\in\Mor(\calC)$ we have
  \begin{equation}
    \label{deri}
    \calX(T)(fg)=\calX(T)(f)g+f\calX(T)(g).
    \end{equation}
    By definition of $\calX$, for any $h\in\Mor(\calC)$ we have $$\calX(T)(h)\subseteq k[\End_k(s(h),t(h))]$$ hence, if $s(f)\neq t(g)$, both sides of \eqref{deri} are zero. If $s(f)\neq t(g)$, \eqref{deri} is equivalent to
    $$
    \sum_{a\in\End_{\calC}(s(f\circ g))}T^{f\circ g\circ a}_{f\circ g}f\circ g\circ a
    =
    \sum_{b\in\End_{\calC}(s(f))}T^{f\circ b}_{f}f\circ b\circ g
    +
    \sum_{a\in\End_{\calC}(s(g))}T^{g\circ a}_{g}f\circ g\circ a.
    $$
    Since $\calC$ is deterministic and cancellative, by Lemma \ref{end_factorization} (c), there is an isomorphism of monoids $\varphi_{g}:\End_{\calC}(s(g))\rightarrow \End_{\calC}(s(f))$ satisfying $$b\circ g=\varphi_{g}(a)\circ g=g\circ a.$$ Hence, \eqref{deri} is equivalent to
    $$
    \sum_{a\in\End_{\calC}(s(g))}T^{f\circ g\circ a}_{f\circ g}f\circ g\circ a
    =
    \sum_{a\in\End_{\calC}(s(g))}
    \left(T^{f\circ \varphi_{g}(a)}_{f}
    +
    T^{g\circ a}_{g}\right)
    f\circ g\circ a,
    $$
   that is, \eqref{deri} is equivalent to
\begin{equation}\label{eq_char}
  T^{f\circ g\circ a}_{f\circ g}
    =
    T^{f\circ \varphi_{g}(a)}_{f}
    +
    T^{g\circ a}_{g}
\end{equation}
  for all $a\in\End_{\calC}(s(g))$. For any $a\in\End_{\calC}(s(g))$ we have
$$T_g^{g\circ a}=T(\alpha),\quad T_{f}^{f\circ \varphi_g(a)}=T(\beta), \quad T_{f\circ g}^{f\circ g \circ \varphi_{f\circ g}(a)}=T(\gamma),$$ 
where $$\alpha=\left(\begin{tikzcd}
    s(g) \arrow[r,"g"] & t(g) \\
    s(g) \arrow[u,"a"]\arrow[r,"g"] & t(g)\arrow[u,"\varphi_g(a)"]
  \end{tikzcd}\right)
  \quad\beta=
  \left(\begin{tikzcd}
    t(g) \arrow[r,"f"] & t(f) \\
   t(g) \arrow[u,"\varphi_g(a)"]\arrow[r,"f"] & t(f)\arrow[u,"\varphi_f(\varphi_g(a))"]
  \end{tikzcd}\right)
  \quad \gamma=
  \left(\begin{tikzcd}
    s(g) \arrow[r,"f\circ g"] & t(f) \\
    s(g) \arrow[u,"a"]\arrow[r,"f\circ g"] & t(f)\arrow[u,"\varphi_{f\circ g}(a)"]
  \end{tikzcd}\right).$$  
Applying Lemma \ref{end_factorization} (c) we obtain $\gamma=\beta\circ  \alpha$. Thus, equation \eqref{eq_char} is true since $T$ is a character. 
In other words, we have $k$-linear maps
    $$
 \calT:\Der^{\gr}(k\calC)\rightarrow \Char(F^{\ad}(\calC))
\quad
\calX: \Char(F^{\ad}(\calC))\rightarrow \Der^{\gr}(k\calC).
$$
The fact that $\calT$ is surjective with section $\calX$ can be verified using  the same arguments as in Proposition \ref{prop_X_section}.
In order to see that this map is an isomorphism we verify that $\calX$ is also a left inverse to $\calT$. We have
$$
\calX(\calT(X))(g)
=
\sum_{a\in\End_{\calC}(s(g))}(\calT(X))^{g\circ a}_{g}g\circ a
=
\sum_{a\in\End_{\calC}(s(g))}X_{g}^{g\circ a}g\circ a
=X(g)
$$
where the first equality follows from the definition of $\calX$ and the last equality is true by Corollary \ref{deterministic_cochains}.
  \end{proof}
  
\textbf{Acknowledgments.} This work was supported by a grant of the Ministry of Research, Innovation and Digitalization, CNCS/CCCDI–UEFISCDI, project number PN-III-P1-1.1-TE-2019-0136, within PNCDI III.

\end{document}